\documentclass[11pt]{amsart}
\pdfoutput=1
\usepackage{amssymb}
\usepackage{amsmath}
\usepackage{amsfonts}
\usepackage[usenames]{color}
\usepackage{graphicx}
\usepackage{array}
\usepackage{psfrag}
\usepackage{color}
\usepackage{ulem}
\usepackage{multirow}

\definecolor{maroon}{RGB}{250, 0, 150}
\definecolor{orange}{RGB}{255, 80, 0}
\definecolor{ogreen}{RGB}{80, 150, 80}

\makeatletter
 
 \@addtoreset{equation}{section}
\makeatother

\newtheorem{theorem}{Theorem}
\newtheorem{lemma}[theorem]{Lemma}

\newtheorem{corollary}[theorem]{Corollary}

\theoremstyle{definition}
\newtheorem{definition}[theorem]{Definition}
\newtheorem{example}[theorem]{Example}

\theoremstyle{remark}

\numberwithin{equation}{section}
\newcommand{\R}{\mathbb{R}}

\newcommand{\C}{\mathcal{C}}

\newcommand{\tr}{\textcolor{red}}

\newcommand{\be}{\begin{enumerate}}
\newcommand{\ee}{\end{enumerate}}

\newcommand{\tbn}{Thurston-Bennequin number }
\newcommand{\rn}{rotation number }

\begin{document}

\title{Legendrian $\theta-$graphs}

\author{Danielle O'Donnol$^ \dagger$}
\address{Department of Mathematics, Imperial College London, London SW7 2AZ,  UK}
\email{d.odonnol@imperial.ac.uk}

\author{Elena Pavelescu}
\address{Department of Mathematics, Oklahoma State University, Stillwater,  OK 74078, USA}
\email{elena.pavelescu@okstate.edu}

%    General info

\thanks{$\dagger$ supported in part by an AMS-Simons Travel Grant}

\date{\today}


\begin{abstract} 
In this article we give necessary and sufficient conditions for two triples of integers to be realized as the \tbn and the \rn of a Legendrian $\theta-$graph with all cycles unknotted.
We show that these invariants are not enough to determine the Legendrian class of a topologically planar $\theta-$graph. 
We define the transverse push-off of a Legendrian graph and we determine its self linking number for  Legendrian $\theta-$graphs.
In the case of topologically planar $\theta-$graphs, we prove that the topological type of the transverse push-off is that of a pretzel link.
\end{abstract}

\maketitle

\section{Introduction}\label{intro}

In this paper, we continue the systematic study of Legendrian graphs in $(\R^3, \xi_{std})$ initiated in \cite{ODPa} .
Legendrian graphs have appeared naturally in several important contexts in the study of contact manifolds.  
They are used in Giroux's proof of existence of open book decompositions compatible with a given contact structure \cite{G}. 
Legendrian graphs also appeared in  Eliashberg and Fraser's proof of the Legendrian simplicity of the unknot \cite{EF}. 

In this article we focus on Legendrian $\theta-$graphs.
We predominantly work with topologically planar embeddings and embeddings where all the cycles are unknots.
In the first part we investigate questions about realizability of the classical invariants and whether the Legendrain type can be determined by these invariants.
In the second part we introduce the transverse push-off a Legendrian graph and investigate its properties in the case of $\theta-$graphs.

In \cite{ODPa}, the authors extended the classical invariants Thurston-Bennequin number, $tb$, and rotation number, $rot$, from Legendrian knots to Legendrian graphs.  
Here we prove that all possible pairs of $(tb, rot)$ for a $\theta-$graph with unknotted cycles are realized.
It is easily shown that  all pairs of integers $(tb, rot)$ of different parities and such that $tb+|rot|\le -1$ can be realized  as the Thurston-Bennequin number and the rotation number of a Legendrian unknot.
We call a pair of integers \textit{acceptable} if they satisfy the two restrictions above.
For $\theta-$graphs, we show the following: 
%the additional restriction $R=rot_1-  rot_2 +rot_3 \in\{0, -1\}$ is sufficient for two triples $rot=(rot_1, rot_2, rot_3)$ and $tb=(tb_1, tb_2, tb_3)$ to be realized as invariants of a Legendrian $\theta-$graph with all cycles unknotted. 

\begin{theorem} 
Any two triples of integers $(tb_1, tb_2, tb_3)$ and $(rot_1, rot_2, rot_3)$  for which  $(tb_i, rot_i)$ are acceptable and  $R=rot_1-  rot_2 +rot_3 \in\{0, -1\}$ can be realized as the Thurston-Bennequin number and the rotation number of a Legendrian $\theta-$graph with all cycles unknotted. 
\end{theorem}

It is known that certain Legendrian knots and links are determined by the invariants $tb$ and $rot$: the unknot  \cite{EF},  torus knots and the figure eight knot \cite{EH}, and links consisting of an unknot and a cable of that unknot  \cite{DG}.
To ask the same question in the context of Legendrian graphs, we restrict to topologically planar Legendrian $\theta-$graphs.
A \textit{topologically planar graph} is  one which is ambient isotopic to a planar embedding.
%A natural question to ask in this context \tma{the context of graphs} is whether the invariants $tb$ and $rot$ determine the Legendrian type of a topologically planar $\theta-$graph. 
The answer is no, the \tbn and the \rn do not determine the Legendrian type of a topologically planar $\theta-$graph. 
The pair of graphs in Figure~\ref{fig-not-determined} provides a counterexample.

The second part of this article is concerned with Legendrian ribbons of Legendrian $\theta-$graphs and their boundary.
Roughly, a ribbon of a Legendrian graph $g$ is a compact oriented surface $R_g$ containing $g$  in its interior, such that the contact structure is tangent to $R_g$ along $g$, transverse to $R_g\smallsetminus g$, and $\partial R_g$ is a transverse knot or link.
We define the \textit{transverse push-off of} $g$ to be the boundary of $R_g $.
This introduces two new invariants of Legendrian graphs, the transverse push-off and its self linking number.
In the case of a Legendrian knot, this definition gives a two component link consisting of both the positive and the negative transverse push-offs.
However, with graphs the transverse push-off can have various numbers of components, depending on connectivity and Legendrian type.

We show the push-off  of a Legendrian $\theta-$graph is either  a transverse knot $K$ with $sl=1$ or a three component transverse link whose three components are the  positive transverse push-offs of the three Legendrian cycles given the correct orientation.
For topologically planar graphs, the topological type of $\partial R_g$ is determined solely by the Thurston-Bennequin number of $g$, as per the following:

\begin{theorem} Let $G$ represent a topologically planar Legendrian $\theta-$graph with $tb = (tb_1, tb_2, tb_3)$.
Then the boundary of its attached ribbon is  an $(a_1, a_2, a_3)-$pretzel, where 
$a_1= tb_1+tb_2-tb_3$,
$a_2= tb_1+tb_3-tb_2$, 
$a_3= tb_2+tb_3-tb_1$.
\label{thm-pretzel}
\end{theorem}

\noindent This elegant relation is specific to $\theta-$graphs and does not generalize to $n\theta-$graphs for $n>3$. 
We give examples to sustain this claim in the last part of the article.
This phenomenon is due to the relationship between flat vertex graphs and pliable vertex graph in the special case of  all vertices of degree at most three. 

\subsection*{Acknowledgements}
The authors would like to thank Tim Cochran and John Etnyre for their continued support, and Chris Wendel and Patrick Massot for helpful conversations.

%---- Section Background ------

\section{Background}
We give a short overview of contact structure,  Legendrian and transverse knots and their invariants. We recall how the invariants of Legendrian knots can be extended to Legendrian graphs.
Let $M$ be an oriented 3-manifold and $\xi$ a 2-plane field on $M$.
If  $\xi =\ker \alpha$ for some $1-$form $\alpha$ on $M$ satisfying $\alpha\wedge d\alpha > 0,$ then $\xi$ is a \textit{contact structure} on $M$.
On $\mathbb{R}^3$, the $1-$form $\alpha = \, dz - y\, dx$ defines a contact structure  called the standard contact structure, $\xi_{std}$.  
Throughout this article we work in $(\mathbb{R}^3, \xi_{std})$.

A knot $K \subset (M, \xi)$ is called \textit{Legendrian} if for all $p\in K$ and $\xi_p$ the contact plane at $p$,  $T_pK\subset \xi_p$.
A spatial graph $G$ is called \textit{Legendrian} if all its edges are Legendrian curves that are non-tangent to each other at the vertices.
If all edges around a vertex are oriented outward, then no two tangent vectors at the vertex coincide in the contact plane.
However, two tangent vectors may have the same direction but different orientations resulting in a smooth arc through the vertex.
It is a result of this structure that the order of the edges around a vertex in a contact plane is not changed up to cyclic permutation under Legendrian isotopy.  
We study Legendrian knots  and graphs via their front projection, the projection on the $xz-$plane. 
Two generic front projections of a Legendrian graph are related by Reidemeister moves I, II and III, together with three moves given by the mutual position of vertices and edges \cite{BI}. See Figure~\ref{moves}.

%---- Figure ---- Reidemeister moves + + + -----------

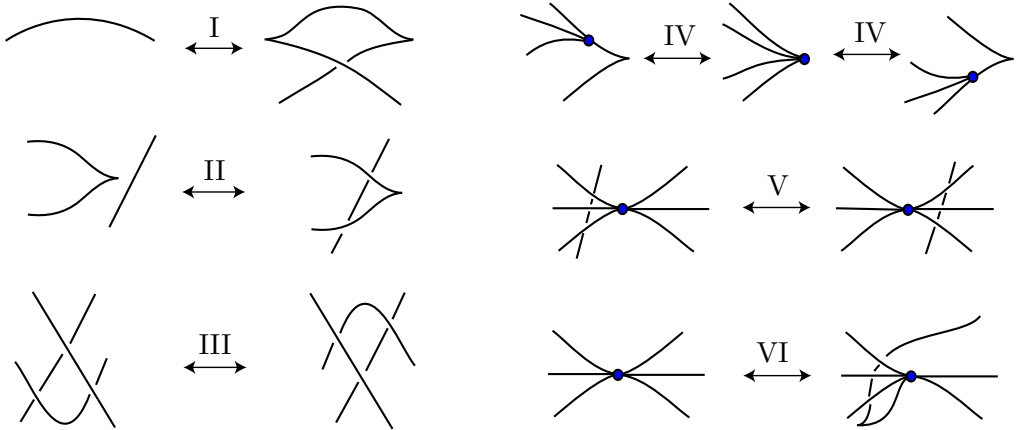
\begin{figure}[htpb!]
\begin{center}
\begin{picture}(400, 188)
\put(0,0){\includegraphics[width=5.5in]{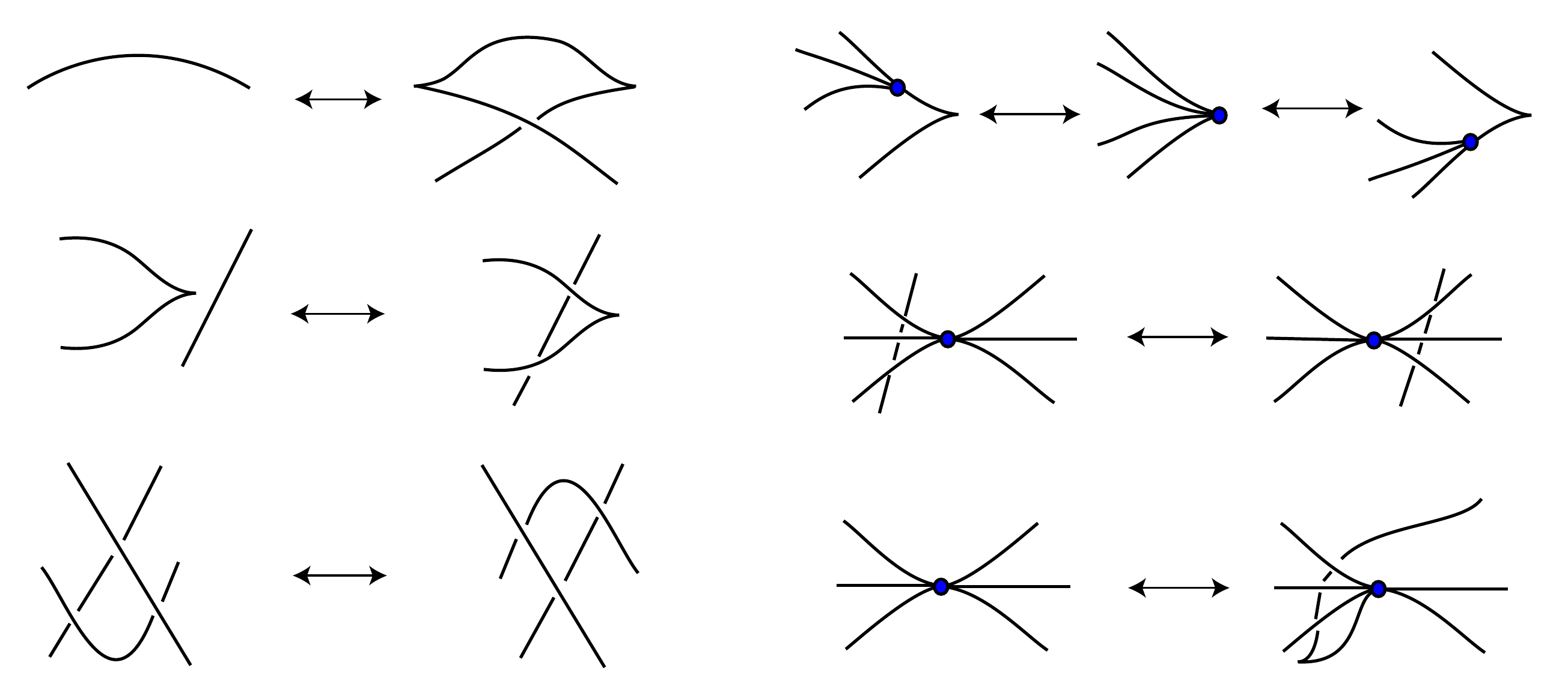}}
\put(84, 153){I}
\put(82,99){II}
\put(80,33){III}
\put(256, 150){IV}
\put(328, 151){IV}
\put(295,93){V}
\put(291,30){VI}
\end{picture}
\caption{\small Legendrian isotopy moves for graphs:  Reidemeister moves I, II and III, a vertex passing through a cusp (IV), an edge passing under or over a vertex (V), an edge adjacent to a vertex rotates to the other side of the vertex (VI). Reflections of these moves that are Legendrian front projections are also allowed.}\label{moves}
\end{center}
\end{figure}

Apart from the topological knot class, there are two classical invariants of Legendrian knots, the Thurston-Bennequin number, $tb$, and the rotation number, $rot$. 
The Thurston-Bennequin number is independent of the orientation on $K$ and measures the twisting of the contact framing on $K$ with respect to the Seifert framing.
To compute the Thurston-Bennequin number of a Legendrian knot $K$, consider a non-zero vector field $v$ transverse to $\xi$, take $K'$ the push-off of $K$ in the direction of $v$, and define $tb(K):= lk(K,K').$ 
For a Legendrian knot $K$,  $tb(K)$ can be computed in terms of the writhe and the number of cusps in its front projection $\tilde{K}$ as $$tb(K) = w(\tilde{K})-\frac{1}{2}\textrm{cusps}(\tilde{K}).$$

To define the rotation number, $rot(K)$, assume $K$ is oriented and $K=\partial \Sigma$, where $\Sigma\subset \mathbb{R}^3$ is an embedded oriented surface. 
The contact planes when restricted to $\Sigma$ form a trivial $2-$dimensional bundle and the trivialization of $\xi | _{\Sigma}$ induces a trivialization on  $\xi | _L = L\times \mathbb{R}^2$.  
Let $v$ be a non-zero vector field tangent to $K$ pointing in the direction of the orientation on $K$.  
The winding number of $v$ about the origin with respect to this trivialization is the rotation number of $K$, denoted $rot(K)$ . 
Taking the positively oriented trivialization $\{d_1=\frac{\partial}{\partial y}, d_2=-y\, \frac{\partial }{\partial z}-\frac{\partial}{\partial x}  \}$ for $\xi_{std}$, one can check that for $\tilde{K}$ the front projection for $K$, 
$$  rot(K) = \frac{1}{2}(\downarrow\textrm{cusps}(\tilde{K})-\uparrow\textrm{cusps}(\tilde{K})).$$

Given a Legendrian knot $K$, Legendrian knots in the same topological class as $K$ can be obtained by stabilizations.
A \textit{stabilization} means replacing a strand of $K$ in the front projection of $K$ by one of the zig-zags  in Figure~\ref{stabilizations}. 
The stabilization is said to be positive if down cusps are introduced and negative if up cusps are introduced. 
The Legendrian isotopy type of $K$ changes through stabilization and so do the Thurston-Bennequin number and rotation number : $tb(S_{\pm}(K)) = tb(K) -1$ and $rot(S_{\pm}(K)) = rot(K) \pm 1$.

% ----------- stabilizations -----------------

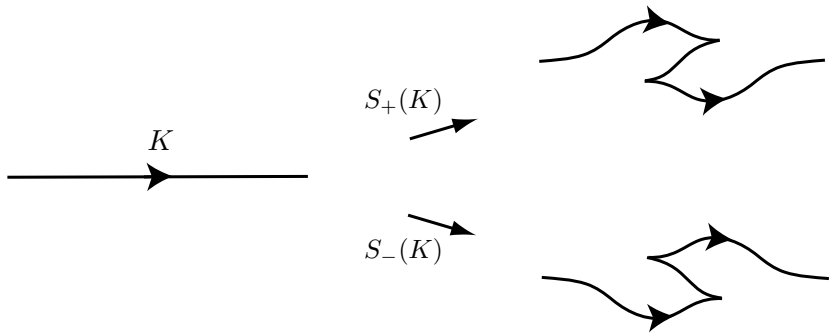
\begin{figure}[htpb!]
\begin{center}
\begin{picture}(300, 119)
\put(0,0){\includegraphics[width=4.3in]{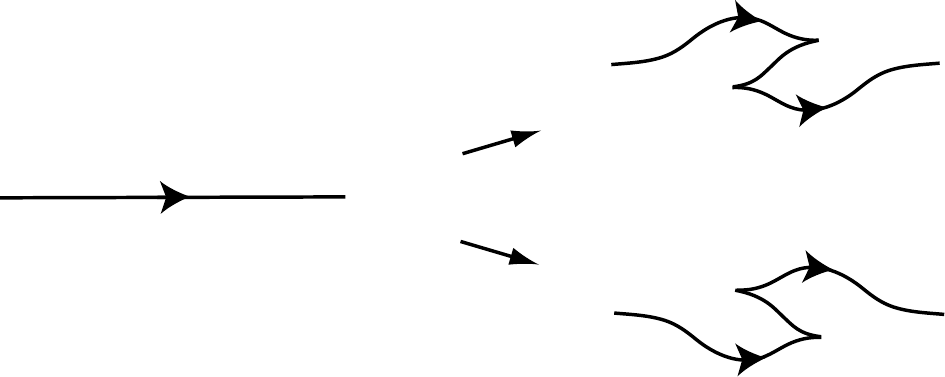}}
\put(53,69){$K$}
\put(135,85){\small $S_+(K)$}
\put(135,28){\small $S_-(K)$}
\end{picture}
\caption{Positive and negative stabilizations in the front
projection.}\label{stabilizations}
\end{center}
\end{figure}

Both the Thurston-Bennequin number and the rotation number can be extended to piece-wise smooth Legendrian knots and to Legendrian graphs \cite{ODPa}.
For a Legendrian graph $G$,  fix an order on the cycles of $G$ and define $tb(G)$ as the ordered list of the Thurston-Bennequin numbers of the cycles of $G$.
Once we fix an order on the cycles of $G$ with orientation, we define  $rot(G)$ to be the ordered list of the rotation numbers of the cycles of $G$.  
If $G$ has no cycles, define both $tb(G)$ and $rot(G)$ to be the empty list.

An oriented knot $t \subset (\mathbb{R}^3, \xi_{std})$ is called \textit{transverse} if for all $p\in t$ and $\xi_p$ the contact plane at $p$,  $T_pt$ is positively transverse to $\xi_p$.
If  $t$ is transverse, we let $\Sigma$ be an oriented surface with $t=\partial \Sigma$.
 As above, $\xi|_\Sigma$ is trivial, so there is a  non-zero vector field $v$ over $\Sigma$ in $\xi$.
 If $t'$ is obtained by pushing $t$ slightly in the direction of $v$, then the \textit{self linking number} of $t$ is $sl(t)= lk(t, t')$.
It is easily seen that  if $\tilde{t}$ is the front projection of $t$, then $sl(t)=writhe(\tilde{t})$.

For an embedded surface $\Sigma \subset (\mathbb{R}^3, \xi_{std})$, the intersection $l_x=T_x\Sigma \cap \xi_x$ is a line for most $x\in \Sigma$, except where the contact plane and the plane tangent to $\Sigma$ coincide.
We denote by $\l:=\cup l_x \subset T\Sigma$ this singular line field, where the union includes lines of intersection only.
Then, there is a singular foliation $\mathcal{F}$, called \textit{the characteristic foliation on $\Sigma$}, whose leaves are tangent to $l$.

% rephrase and move to later in the paper Through Legendrian isotopy the order of edges in the contact plane at the vertex of  a Legendrian graph does not change.
%Thus Legendrian graphs can be considered a special case of flat vertex graphs. 
%We are using this perspective in the later sections of the paper and it is useful to recall the Reidemeister moves for  pliable and flat vertex graphs  \cite[pages 699, 704]{K}.
%Apart from the standard three Reidemeister moves for knots, there are two moves due to the relative positions of edges and vertices, as in Figure~\ref{fig-ReidemeisterIVandV.

\section{realization theorem}
In this section we find which triples of integers can be realized as $tb$ and $rot$ of  Legendrian $\theta-$graphs with all cycles unknotted.
Both the structure of the $\theta-$graph and the required unknotted cycles impose restrictions on these integers.
We also investigate whether $tb$ and $rot$ uniquely determine the Legendrian type.
The following lemma identifies restrictions on the invariants of Legendrian unknots.

\begin{lemma} All pairs of integers $(tb, rot)$ of different parities and such that $$tb+|rot|\le -1$$ can be realized  as the Thurston-Bennequin number and the rotation number of a Legendrian unknot.
\end{lemma}
\begin{proof}
We know from  \cite{El} that for a Legendrian unknot $K$ in $(\mathbb{R}^3, \xi_{std})$, $tb(K)+|rot(K)|\le -1$. 
Eliashberg and Fraser \cite{EF} showed that a Legendrian unknot $K$ is Legendrian isotopic to a unique unknot in standard form.
The standard forms are shown in  Figure~\ref{fig-standard_unknotEF}.
The front projection in Figure~\ref{fig-standard_unknotEF}(a) represents two distinct Legendrian classes, depending on the chosen orientation. 
For the front projection shown in Figure~\ref{fig-standard_unknotEF}(b) both orientations give the same Legendrian class. 
The number of cusps and the number of crossings of the unknot in standard form are uniquely determined by $tb(K)$ and $rot(K)$ as follows:
\be
\item If $rot(K)\ne 0$ (Figure~\ref{fig-standard_unknotEF}(a)), then 
$$ tb(K) = -(2t+1+s)  $$
$$ rot(K) =  \left\{
\begin{array}{rl}
s, & \mbox{ if  the leftmost cusp is a down cusp} \\
-s, & \mbox{ if the leftmost cusp is an up cusp.}  \end{array}
\right.  $$
\item If $rot(K) =0$  (Figure~\ref{fig-standard_unknotEF}(b)), then 
$$ tb(K) = -(2t+1).  $$
\ee
Notice that in both cases the $tb$ and $rot$ have different parities.  

%----- Figure -----------
\begin{figure}[htpb!]
\begin{center}
\begin{picture}(330, 70)
\put(0,0){\includegraphics{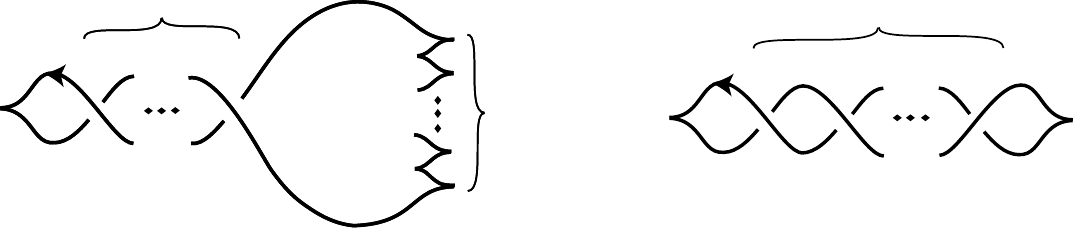}}
\put(55,-6){(a)}
\put(249, -6){(b)}
\put(35,65){\small $2t+1$}
\put(145, 31){\small $s$}
\put(250,63){\small $2t$}
\end{picture}
\caption{\small Legendrian unknot in standard form: (a) $rot (K) > 0$ [reverse orientation gives $rot(K)< 0$], (b) $rot(K) =0$.}
\label{fig-standard_unknotEF}
\end{center}
\end{figure}
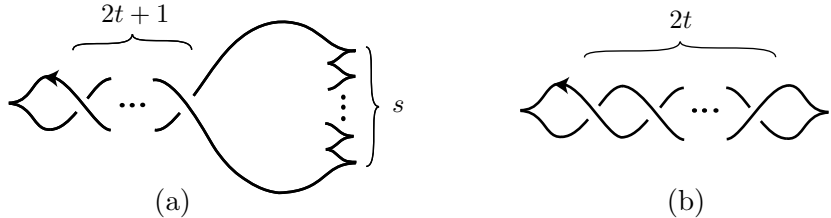

For a pair $(tb, rot)$, the integers  $s$ and $t$ are determined as follows: 
\begin{itemize}
\item If $rot>0$, the pair $(tb, rot)$ is realized via the Legendrian unknot with front projection as in Figure~\ref{fig-standard_unknotEF}(a), for $(t,s)=(-\frac{tb+rot+1}{2}, rot)$.
\item If $rot<0$, the pair $(tb, rot)$ is realized via the Legendrian unknot with front projection as in Figure~\ref{fig-standard_unknotEF}(a), for $(t,s)=(-\frac{tb-rot+1}{2}, -rot)$.
\item If $rot=0$, the pair $(tb, rot)$ is realized via the Legendrian unknot with front projection as in Figure~\ref{fig-standard_unknotEF}(b), for $t=-\frac{tb+1}{2}$.
\end{itemize}

\end{proof}

\noindent We have described the pairs $(tb, rot)$ that can occur for the unknot.  

Towards the proof of  Theorem \ref{theorem-realize}, we show in the next lemma that Legendrian $\theta-$graphs can be standardized near their two vertices.

\begin{lemma}\label{NearVer}
Any Legendrian $\theta-$graph $G$, can be Legendrian isotoped to a graph $\tilde{G}$ whose front projection looks as in Figure ~\ref{fig-two-vertices} in the neighborhood of its two vertices. 
\end{lemma}

\begin{proof}
Label the vertices of $G$ by $a$ and $b$. 
In the front projection of $G$, use the Reidemeister VI move if necessary, to move the three strands on the right of vertex $a$ while near $a$ and on the left of vertex $b$ while near $b$. 
Then, small enough neighborhoods of the two vertices look as in Figure ~\ref{fig-two-vertices}.
\end{proof}

%----- Figure -----------
\begin{figure}[htpb!]
\begin{center}
\begin{picture}(200, 70)
\put(0,0){\includegraphics{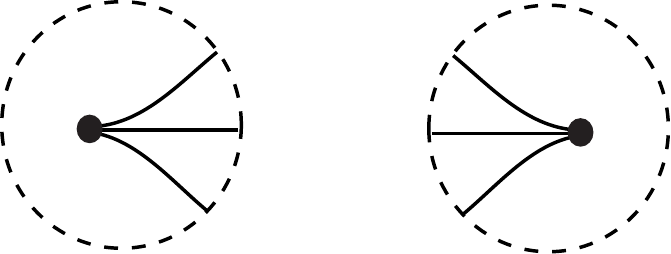}}
\put(15, 33){\small $a$}
\put(173, 32){\small $b$}
\end{picture}
\caption{Legendrian $\theta-$graph near its two vertices.} 
\label{fig-two-vertices}
\end{center}
\end{figure}

For the remainder of this section, we assume that near its two vertices, $a$ and $b$, the front projection of the graph looks as in Figure~\ref{fig-two-vertices}.
We  fix notation: $e_1$ is the top strand at $a$ in the front projection,  $e_2$ is the middle strand at $a$, $e_3$ is the lower strand at $a$,  
 $\C_1$ is  the oriented cycle exiting vertex $a$ along $e_1$ and entering  vertex $a$ along $e_2$,
$\C_2$ is the oriented cycle exiting vertex $a$ along $e_1$ and entering  vertex $a$ along $e_3$,
$\C_3$ is the oriented cycle exiting vertex $a$ along $e_2$ and entering vertex $a$ along $e_3$. 
We note that there is no consistent way of orienting the three edges which gives three oriented cycles.  
It should also be noted that the above notation is a labelling given after the graph is embedded.  
If a labelled graph is embedded relabelling of the graph and reorienting of the cycles may be necessary in order to have the following lemma apply.  

In the next lemma we show what additional restrictions occur as a result of the structure of the $\theta-$graph.
\begin{lemma}
Let $rot_1,  rot_2$ and $rot_3$ be integers representing rotation numbers for cycles $\C_1$, $\C_2$ and $\C_3$, in the above notation.
Then $rot_1-  rot_2 +rot_3 \in \{0, -1\}$.
 \end{lemma}

\begin{proof}
For $i=1, 2, 3$, let $k_i$ ($k_i'$) represent the number of positive (negative) stabilizations along the edge $e_i$ when oriented from vertex $a$ to vertex $b$.
Let $s_i:=k_i-k_i'$, for $i=1, 2, 3$.
Then,

$$ rot_1 =  \left\{
\begin{array}{rl}
s_1-s_2, & \mbox{ if $\C_1$ has a down cusp at $b$} \\
s_1-s_2-1, & \mbox{  if $\C_1$ has an up cusp at $b$}  \end{array}
\right.  $$

$$ rot_2 =  \left\{
\begin{array}{rl}
s_1-s_3, & \mbox{  if $\C_2$ has a down cusp at $b$} \\
s_1-s_3-1, & \mbox{ if $\C_2$ has an up cusp at $b$}  \end{array}
\right.  $$

$$ rot_3 =  \left\{
\begin{array}{rl}
s_2-s_3, & \mbox{ if $\C_3$ has a down cusp at $b$} \\
s_2-s_3-1, & \mbox{  if $\C_3$ has an up cusp at $b$}  \end{array}
\right.  $$

This gives eight different possible combinations and the possible values of  $R=rot_1-  rot_2 +rot_3$ are given in Table 1.

\begin{table}[htdp]
\begin{center}
\begin{tabular}{|c|c|c|c|r|}
\hline
\multirow{2}{*}{Case} & \multicolumn{3}{|c|}{Cusp at $b$} & \multirow{2}{*}{$R$}\\ \cline{2-4}
  & $\C_1$ & $\C_2$ & $\C_3$ & \\
  \hline
 1 & $\downarrow$ & $\downarrow$ & $\downarrow$ & 0 \\
  \hline
   2 & $\downarrow$ & $\downarrow$ & $\uparrow$ & $-1$ \\
 \hline
 3 & $\downarrow$ & $\uparrow$ & $\downarrow$ & 1 \\
 \hline
 4 & $\downarrow$ & $\uparrow$ & $\uparrow$ & 0 \\
 \hline

 5 & $\uparrow$ & $\downarrow$ & $\downarrow$ & $-1$ \\
 \hline
  6 & $\uparrow$ & $\downarrow$ & $\uparrow$ & $-2$ \\
 \hline

 7 & $\uparrow$ & $\uparrow$ & $\downarrow$ & 0 \\
 \hline
 8 & $\uparrow$ & $\uparrow$ & $\uparrow$ & $-1$ \\
\hline
\end{tabular}
\end{center}
\label{table-sums}
\caption{Possible values for $R=rot_1-rot_2+rot_3.$}
\end{table}

Case 3 cannot occur.  If both $\C_1$ and $\C_3$ have a down cusp at $b$, edge $e_3$ sits below edge $e_1$ at $b$, hence $\C_2$ has a down cusp at $b$.
Also, Case 6 cannot occur.  If both $\C_1$ and $\C_3$ have an up cusp at $b$, edge $e_3$ sits above edge $e_1$ at $b$, hence $\C_2$ has an up cusp at $b$.
Among the six remaining cases, three of them give $rot_1-  rot_2 +rot_3 =-1$ (when there is an odd number of up cusps at $b$ between $\C_1$, $\C_2$ and $\C_3$).
The other three cases give $rot_1-  rot_2 +rot_3 = 0$ (when there is an even number of up cusps at $b$ between $\C_1$, $\C_2$ and $\C_3$). 
\end{proof}

\begin{theorem}
Any two triples of integers $(tb_1, tb_2, tb_3)$ and $(rot_1, rot_2, rot_3)$  for which  $tb_i+|rot_i|\le -1$, $tb_i$ and $rot_i$ are of different parities for $i=1,2,3$ and  $R=rot_1-  rot_2 +rot_3 \in\{0, -1\}$ can be realized as the Thurston-Bennequin number and the rotation number of a Legendrian $\theta-$graph with all cycles unknotted.
\label{theorem-realize}
\end{theorem}

\begin{proof} Let $tb=(tb_1, tb_2, tb_3)$ and $rot=(rot_1, rot_2, rot_3)$  be triples of integers as in the hypothesis.
We give front projections of Legendrian $\theta-$graphs realizing these triples.
Let $r_i:= |rot_i|$, for $i=1, 2, 3$.
We differentiate our examples according to the values of $rot_1$, $rot_2$ and $rot_3$ and the relationship between $r_1, r_2$ and $r_3$.

\begin{table}[htdp]
\begin{center}
\begin{tabular}{|c|c|c|c|c|c|}
\hline
Case & $rot_1$ & $rot_2$ & $rot_3$ & $R=0$ & $R=-1$\\ \hline
(i) & + & + & + & $r_1-r_2+r_3=0$ & $r_1+r_3+1=r_2$\\
  \hline
  (ii) & + & + &  $-$& $r_1-r_2-r_3=0$ & $r_1+1=r_2+r_3$\\
  \hline
(iii) & + & $-$ & + & $r_1+r_2+r_3=0$ & $r_1+r_2 +r_3 +1=0$\\
  \hline
(iv) & + & $-$ & $-$& $r_1+r_2-r_3=0$ & $r_1+r_2+1=r_3$\\
  \hline
(v) & $-$ & + & + & $-r_1-r_2+r_3=0$ & $r_1+r_2=r_3+1$\\
  \hline
(vi) & $-$ & + & $-$ & $-r_1-r_2-r_3=0$ & $r_1+r_2+r_3=1$\\
  \hline
(vii) & $-$ & $-$ & + & $-r_1+r_2+r_3=0$ & $r_1=r_2+r_3+1$\\
  \hline
(viii) & $-$ &  $-$ & $-$ & $-r_1+r_2-r_3=0$ & $r_1+r_3=r_2+1$\\
  \hline
\end{tabular}
\end{center}
\label{sums1}
\caption{ $+$ stands for $rot_i\ge 0$ and $-$ stands for $rot_i<0$}

\end{table}

When $R=0$, for each case (i)--(viii), there is a choice of indices $i,j,k$ with  $\{i,j,k\}=\{1,2,3\}$
 such that $r_i \ge r_j+r_k$ (in fact, $r_i = r_j+r_k$).

When $R=-1$, for each case (i), (iv), (vi) and (vii) there is a choice of indices $i,j,k$ with $\{i,j,k\}=\{1,2,3\}$  such that $r_i \ge r_j+r_k$;  
case (iii) is not realized; and for each case  (ii), (v) and (viii), there is a choice of indices $i,j,k$ with  $\{i,j,k\}=\{1,2,3\}$ such that $r_i + 1=r_j+r_k$. 

Thus any realizable  $(rot_1, rot_2, rot_3)$ falls under at least one of the following six conditions:  (1) $r_1 \ge r_2+r_3$,  (2) $r_2 \ge r_1+r_3$, (3) $r_3 \ge r_1+r_2$, (4) $r_1 + 1=r_2+r_3$, (5) $r_2 + 1=r_1+r_3$ and (6) $r_3 + 1=r_1+r_2$.
We describe ways of realizing the invariants for these six cases.

The cycles $\C_1, \C_2$ and $\C_3$ are as described earlier.   
The choice of orientations for the three cycles implies that $e_1$ is oriented from $a$ to $b$ in both $\C_1$ and $\C_2$, while  $e_3$ is oriented from $b$ to $a$ in both $\C_2$ and $\C_3$.
A box along a single strand designates number of stabilizations along the strand.
We take

\begin{itemize}
\item $r_i$ positive stabilizations if $rot_i\ge 0$
\item $r_i$ negative stabilizations if $rot_i< 0$,
\end{itemize}
when edges $e_1$, $e_2$ and $e_3$ are oriented as in cycle $\C_i$.
A box along a pair of strands designates number of crossings between the two strands.
All the crossings are as those in Figure~\ref{fig-standard_unknotEF}.\\

\noindent {\bf Case 1.}$(\mathbf{ r_1\ge r_2+r_3})$  Figure~\ref{fig-r1} represents the front projection of a Legendrian $\theta-$graph with the prescribed $tb$ and $rot$.

%----- Figure -----------

\begin{figure}[htpb!]
\begin{center}
\begin{picture}(380, 40)
\put(0,0){\includegraphics{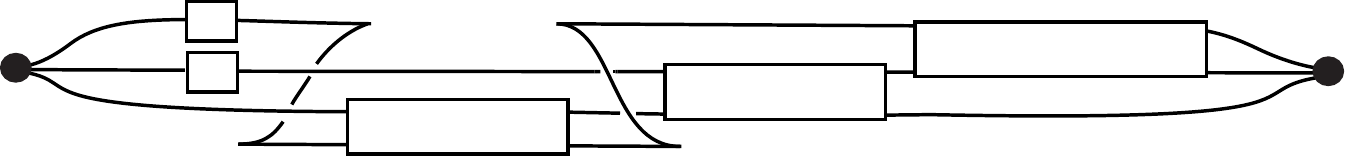}}
\put(57,37){\small $r_2$}
\put(57, 22){\small $r_3$}
\put(37,43){\small $e_1$}
\put(37,28){\small $e_2$}
\put(37,7){\small $e_3$}
\put(217,42){\small $e_1$}
\put(307,5){\small $e_3$}
\put(103,6){\small $-tb_2-r_2-1$}
\put(195,16){\small $-tb_3-r_3-1$}
\put(266,28){\small $-tb_1-r_2-r_3-1$}
%\put(20, 52){\tiny $e_1$}
%\put(23, 38){\tiny $e_2$}
%\put(20, 16){\tiny $e_3$}
%\put(170, 47){\tiny $e_1$}
%\put(290, 34){\tiny $e_2$}
\end{picture}
\caption{\small  
%Legendrian $\theta-$graph with $tb= (tb_1, tb_2, tb_3)$ and $rot=(rot_1, rot_2, rot_3)$.
Case 1: $ r_1\ge r_2+r_3$.
}\label{fig-r1}
\end{center}
\end{figure}
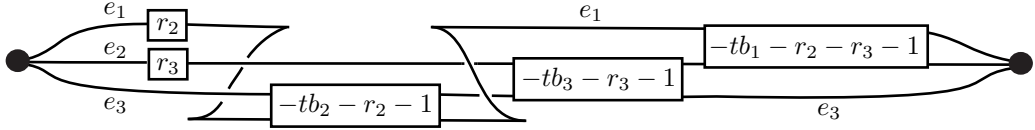

 \noindent Since $tb_i+|rot_i|\le -1$, the integers $-tb_2-r_2-1$ and $-tb_3-r_3-1$ are non-negative. 
 Since $ r_1\ge r_2+r_3$, then $-tb_1-r_2-r_3-1\ge -tb_1-r_1-1\ge 0$. 
 So all of the indicated number of half  twists are non-negative integers as needed.  
 The number $-tb_1-r_2-r_3-1$  changes parity, according to whether $rot_1-  rot_2 +rot_3$ equals $-1$ or $0$.
  
We check that the Thurston-Bennequin number and the rotation number for this embedding have the correct values.
For a cycle $\C$ we use $$ tb(\C)= w(\C)-\frac{1}{2}\textrm{cusps}(\C)$$
$$  rot(\C) = \frac{1}{2}(\downarrow\textrm{cusps}(\C)-\uparrow\textrm{cusps}(\C)),$$
 where $w=$ writhe, $\textrm{cusps}=$ total number of cusps, $\downarrow\textrm{cusps}=$ number of down cusps, $\uparrow\textrm{cusps}=$ number of up cusps.
 \begin{itemize}
 \item $ tb(\C_1)= w(\C_1)-\frac{1}{2}\textrm{cusps}(\C_1)= (tb_1+r_2+r_3+3)-(r_2+r_3+3)=tb_1$
 \item $ tb(\C_2)= w(\C_2)-\frac{1}{2}\textrm{cusps}(\C_2) =  (tb_2+r_2+3)-(r_2+3)=tb_2$
  \item $ tb(\C_3)= w(\C_3)-\frac{1}{2}\textrm{cusps}(\C_3)= (tb_3+r_3+1)-(r_3+1)=tb_3$
 \end{itemize} 
 
 %------ Case C=0----------
 
\noindent If $rot_1-  rot_2 +rot_3=0$, then $-tb_1-r_2-r_3-1$ has the same parity as $-tb_1-r_1-1$.
 They are both even, since $tb_1$ and $rot_1$ have different parities.
This implies that at vertex $b$ the upper strand is $e_1$ and the middle strand is $e_2$. 
 \begin{itemize}
 \item $ rot(\C_1)=\frac{1}{2}(\downarrow\textrm{cusps}(\C_1)-\uparrow\textrm{cusps}(\C_1))$\\
$=\frac{1}{2}(2\cdot {\rm sgn}(rot_2)\cdot r_2 +3- 2\cdot {\rm sgn}(rot_3)\cdot r_3 -3) = rot_2-rot_3=rot_1$
 
  \item $ rot(\C_2) = \frac{1}{2}(\downarrow\textrm{cusps}(\C_2)-\uparrow\textrm{cusps}(\C_2))=    \frac{1}{2}(2\cdot {\rm sgn}(rot_2)\cdot r_2 + 3 - 3)=rot_2 $
 \item  $rot(\C_3) = \frac{1}{2}(\downarrow\textrm{cusps}(\C_3)-\uparrow\textrm{cusps}(\C_3)) = \frac{1}{2}(2\cdot  {\rm sgn}(rot_3)\cdot r_3 +1 - 1) =rot_3 $\\
 \end{itemize}
 
  %------ Case C=-1----------

 \noindent If $rot_1-  rot_2 +rot_3=-1$, then$-tb_1-r_2-r_3-1$ has different parity than $-tb_1-r_1-1$. 
Since $tb_1$ and $rot_1$ have different parities, $-tb_1-r_1-1$ is even and $-tb_1-r_2-r_3-1$ is odd.
This implies that at vertex $b$ the upper strand is $e_2$ and the middle strand is $e_1$. 
 Computations for $ rot(\C_2) $ and $ rot(\C_3) $ are the same as above.
 \begin{itemize}
  \item $ rot(\C_1) =\frac{1}{2}(\downarrow\textrm{cusps}(\C_1)-\uparrow\textrm{cusps}(\C_1))$\\
  $= \frac{1}{2}( 2\cdot {\rm sgn}(rot_2)\cdot r_2 +2-  2\cdot {\rm sgn}(rot_3)\cdot r_3 -4) = rot_2-rot_3 -1 =rot_1  $\\
 \end{itemize} 
 
 In the remaining cases, a similar check may be done to verify that they have the correct $tb$ and $rot$.
 
 %------ Case 2----------

\noindent {\bf Case 2.}$\mathbf{(r_2\ge r_1+r_3)}$ Figure~\ref{fig-r2} represents the front projection of a Legendrian $\theta-$graph with the prescribed $tb$ and $rot$.  
Since $r_2\ge r_1+r_3$, then $-tb_2-r_1-r_3-1\ge -tb_2-r_2-1\ge 0$. 
 
 %----- Figure -----------

\begin{figure}[htpb!]
\begin{center}
\begin{picture}(410, 60)
\put(0,0){\includegraphics{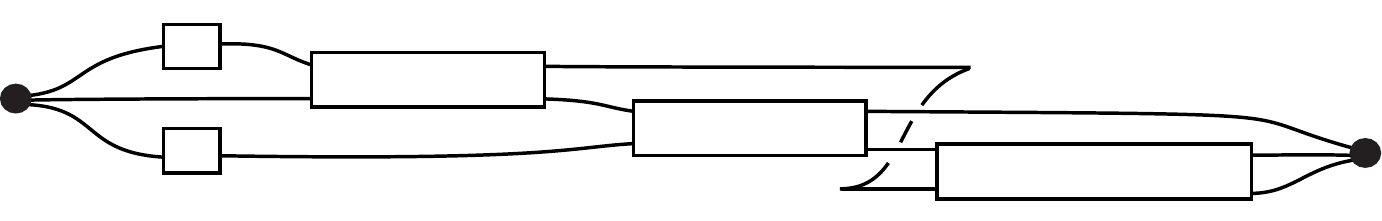}}
\put(51,48){\small $r_1$}
\put(52, 18){\small $r_3$}
\put(96,38){\small $-tb_1-r_1-1$}
\put(188,24){\small $-tb_3-r_3-1$}
\put(276,12){\small $-tb_2-r_1-r_3-1$}
%\put(25, 49){\tiny $e_1$}
%\put(26, 38){\tiny $e_2$}
%\put(25, 18){\tiny $e_3$}
%\put(200, 48){\tiny $e_1$}
%\put(330, 34){\tiny $e_2$}
%\put(281,24){\scriptsize}
\end{picture}
\caption{\small  
%Legendrian $\theta-$graph with $tb= (tb_1, tb_2, tb_3)$ and $rot=(rot_1, rot_2, rot_3)$
Case 2: $r_2\ge r_1+r_3.$
}
\label{fig-r2}
\end{center}
\end{figure}
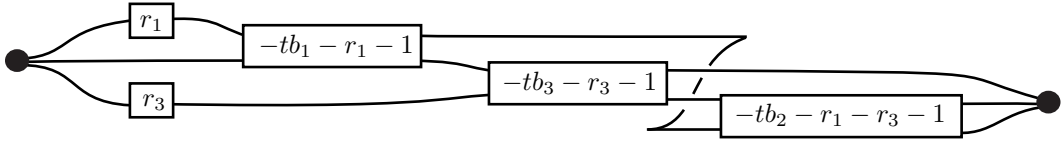

\noindent {\bf Case 3.}$\mathbf{(r_3\ge r_1+r_2)}$ Figure~\ref{fig-r3} represents the front projection of a Legendrian $\theta-$graph with the prescribed $tb$ and $rot$.  
As $r_3\ge r_1+r_2$, then $-tb_3-r_1-r_2-1\ge -tb_3-r_3-1\ge 0$.

 %----- Figure -----------

\begin{figure}[htpb!]
\begin{center}
\begin{picture}(410, 55)
\put(0,0){\includegraphics{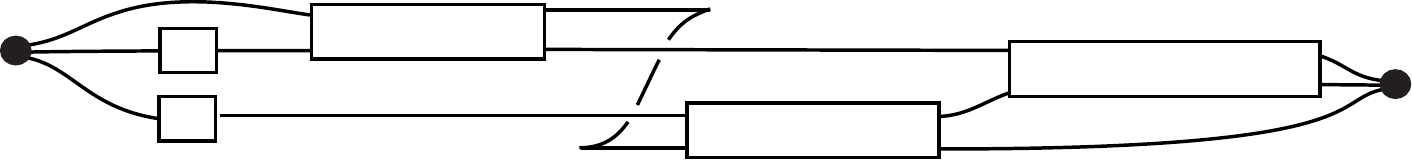}}
\put(50,30){\small $r_1$}
\put(50, 10){\small $r_3$}
\put(94,35){\small $-tb_1-r_1-1$}
\put(205,6){\small $-tb_2-r_2-1$}
\put(295,24){\small $-tb_3-r_1-r_2-1$}
%\put(20, 76){\tiny $e_1$}
%\put(23, 60){\tiny $e_2$}
%\put(20, 39){\tiny $e_3$}
%\put(180, 70){\tiny $e_1$}
%\put(330, 56){\tiny $e_2$}
%\put(350, 6){\tiny $e_1$}
\end{picture}
\caption{\small 
% Legendrian $\theta-$graph with $tb= (tb_1, tb_2, tb_3)$ and $rot=(rot_1, rot_2, rot_3)$
Case 3: $r_3\ge r_1+r_2.$
}
\label{fig-r3}
\end{center}
\end{figure}
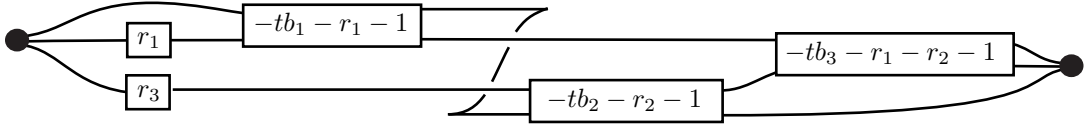

\noindent {\bf Case 4.}$\mathbf{(r_1+1=r_2+r_3)}$  In this case the graph in Figure~\ref{fig-B} realizes $(tb, rot)$. 
Since $r_2+r_3=r_1+1$, we have $-tb_1-r_2-r_3=-tb_1-r_1-1\ge 0$.

  %----- Figure B -----------

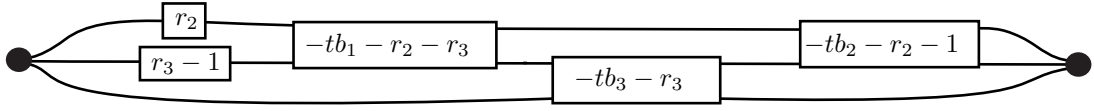
\begin{figure}[htpb!]
\begin{center}
\begin{picture}(430, 55)
\put(0,0){\includegraphics{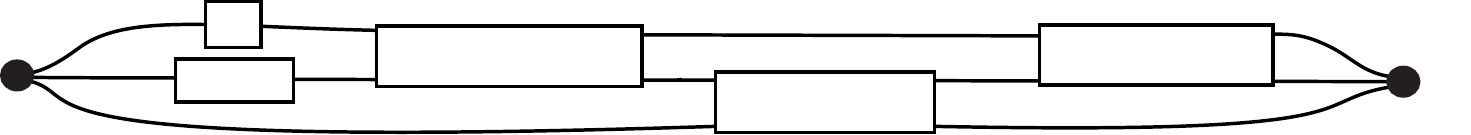}}
\put(64,30){\small $r_2$}
\put(55, 13){\small $r_3-1$}
\put(113,20){\small $-tb_1-r_2-r_3$}
\put(215,7){\small $-tb_3-r_3$}
\put(302,20){\small $-tb_2-r_2-1$}
%\put(20, 76){\tiny $e_1$}
%\put(23, 60){\tiny $e_2$}
%\put(20, 39){\tiny $e_3$}
%\put(180, 70){\tiny $e_1$}
%\put(330, 56){\tiny $e_2$}
%\put(350, 6){\tiny $e_1$}
\end{picture}
\caption{\small Case 4: $r_1+1=r_2+r_3.$
}
\label{fig-B}
\end{center}
\end{figure}

\noindent {\bf Case 5.}$\mathbf{(r_2+1=r_1+r_3)}$  For this case the graph in Figure~\ref{fig-A} realizes $(tb, rot)$.  Given $r_1+r_3=r_2+1$, we have that $-tb_2-r_1-r_3+1=-tb_2-r_2> 0$.

 %----- Figure A -----------

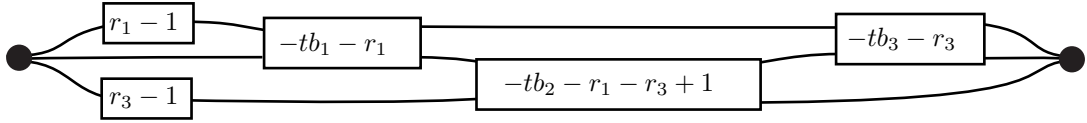
\begin{figure}[htpb!]
\begin{center}
\begin{picture}(430, 55)
\put(0,0){\includegraphics{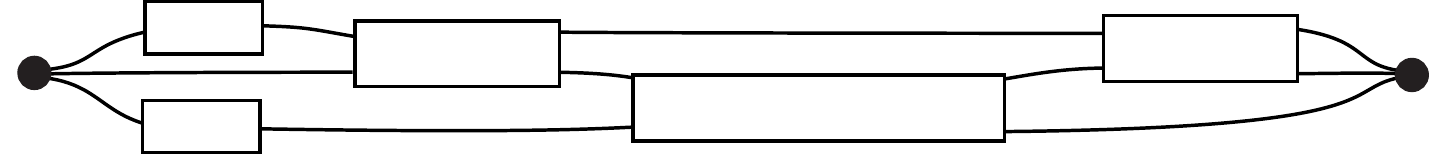}}
\put(44,34){\small $r_1-1$}
\put(44, 6){\small$r_3-1$}
\put(108,26){\small$-tb_1-r_1$}
\put(193,11){\small $-tb_2-r_1-r_3+1$}
\put(323,28){\small $-tb_3-r_3$}
%\put(20, 76){\tiny $e_1$}
%\put(23, 60){\tiny $e_2$}
%\put(20, 39){\tiny $e_3$}
%\put(180, 70){\tiny $e_1$}
%\put(330, 56){\tiny $e_2$}
%\put(350, 6){\tiny $e_1$}
\end{picture}
\caption{\small Case 5: $r_2+1=r_1+r_3.$
}
\label{fig-A}
\end{center}
\end{figure}

\noindent {\bf Case 6.}$\mathbf{(r_3+1=r_1+r_2)}$  In this case the graph in Figure~\ref{fig-C} realizes $(tb, rot).$
Since $r_1+r_2=r_3+1$, we have $-tb_3-r_1-r_2=-tb_3-r_3-1\ge 0$.

 %----- Figure C -----------

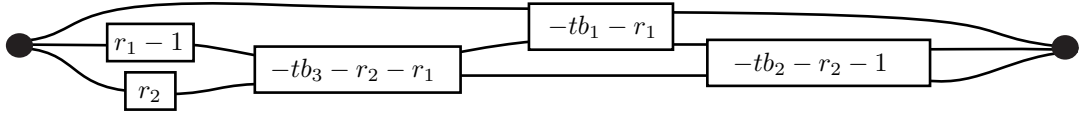
\begin{figure}[htpb!]
\begin{center}
\begin{picture}(4300, 55)
\put(0,0){\includegraphics{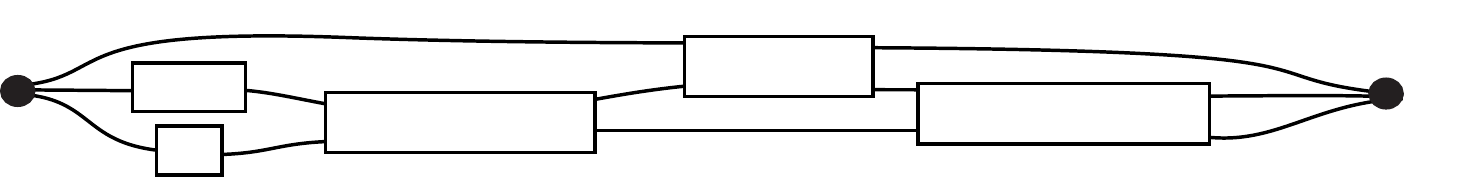}}
\put(41,23){\small  $r_1-1$}
\put(50, 5){\small $r_2$}
\put(100,13){\small $-tb_3-r_2-r_1$}
\put(205,29){\small $-tb_1-r_1$}
\put(275,15){\small $-tb_2-r_2-1$}
%\put(20, 76){\tiny $e_1$}
%\put(23, 60){\tiny $e_2$}
%\put(20, 39){\tiny $e_3$}
%\put(180, 70){\tiny $e_1$}
%\put(330, 56){\tiny $e_2$}
%\put(350, 6){\tiny $e_1$}
\end{picture}
\caption{\small Case 6: $r_3+1=r_1+r_2.$
}
\label{fig-C}
\end{center}
\end{figure}

This completes the proof.  
\end{proof}

%--------tb, rot do not determine Leg type ---------

\subsection{Topologically planar $\theta-$graphs are not Legendrian simple}We ask whether the invariants $tb$ and $rot$ determine the Legendrian type of a planar $\theta-$graph. 
If we do not require that the cyclic order of the edges around the vertex $a$ (or $b$) is the same in both embeddings, the answer is negative.  The following is a counterexample in this case.  

\begin{example}
The two graphs in Figure~\ref{fig-not-determined} have the same invariants but they are not Legendrian isotopic.
Let $\C_1$, $\C_2$ and $\C_3$ be the three cycles of $G$ determined by  the pairs of edges \{$e_1$, $e_2$\}, \{$e_1$, $e_3$\} and  \{$e_2$, $e_3$\}, respectively.
Let $\C'_1$, $\C'_2$ and $\C'_3$ be the three cycles of $G'$ determined by  \{$f_2$, $f_1$\}, \{$f_2$, $f_3$\} and  \{$f_1$, $f_3$\}, respectively.
The cycles have $tb(\C_1)=tb(\C'_1)=-1$, $tb(\C_2)=tb(\C'_2)=-5 $,  $tb(\C_3)=tb(\C'_3)=-3$ and $rot(\C_i)=rot(\C'_i)=0$ for $i=1,2,3$.

Assume the two graphs are Legendrian isotopic. 
Since the cycles with same invariants should correspond to each other via the Legedrian isotopy (which we denote by $\iota$), the edges correspond as $e_1\leftrightarrow \iota(e_1)=f_2$,  $e_2\leftrightarrow \iota(e_2)=f_1$ and $e_3\leftrightarrow \iota(e_3)=f_3$.
But at both vertices of $G$ the (counterclockwise) order of edges in the contact plane is $e_1-e_2-e_3$ and at both vertices of $G'$ the (counterclockwise) order of edges in the contact plane is $\iota(e_1)-\iota(e_3)-\iota(e_2)$. 
This leads to a contradiction, since a Legendrian isotopy preserves the cyclic order of edges at each vertex.
\label{example-not-determined}
\end{example}

%----- Figure not-determined -----------

\begin{figure}[htpb!]
\begin{center}
\begin{picture}(390, 55)
\put(0,0){\includegraphics{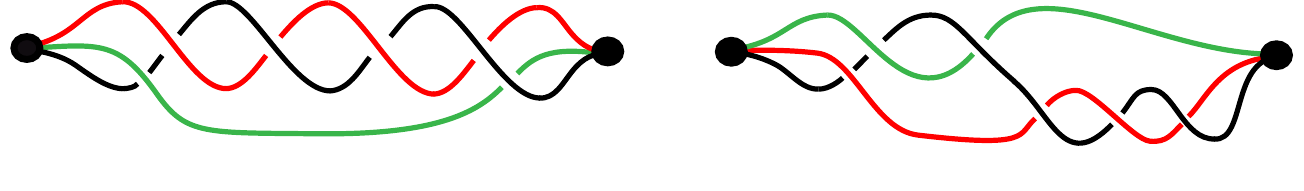}}
\put(80,-5){$G$}
\put(300, -5){$G'$}
\put(19,54){\tiny $e_1$}
\put(27, 43){\tiny$e_2$}
\put(20,27){\tiny$e_3$}
\put(230,55){\tiny $f_1$}
\put(233,43){\tiny $f_2$}
\put(225,24){\tiny $f_3$}
%\put(108,26){\small$-tb_1-r_1$}
%\put(193,11){\small $-tb_2-r_1-r_3+1$}
%\put(323,28){\small $-tb_3-r_3$}
%\put(20, 76){\tiny $e_1$}
%\put(23, 60){\tiny $e_2$}
%\put(20, 39){\tiny $e_3$}
%\put(180, 70){\tiny $e_1$}
%\put(330, 56){\tiny $e_2$}
%\put(350, 6){\tiny $e_1$}
\end{picture}
\caption{\small Non-Legendrian isotopic graphs with the same invariants. 
}
\label{fig-not-determined}
\end{center}
\end{figure}
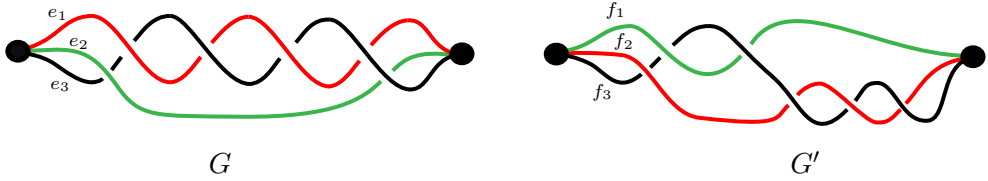

\begin{corollary}
The invariants $tb$ and $rot$ are not enough to distinguish the Legendrian class of an $n\theta-$graph for $n\ge 3$. 
\end{corollary}

\begin{proof}
For $n\ge 4$, a pair of graphs with the same invariants but of different Legendrian type can be otained from $(G, G')$ in Example \ref{example-not-determined} by adding $n-3$ unknotted edges at the top of the three existing ones.
\end{proof}

\section{Legendrian Ribbons and transverse push-offs}
\label{sec-ribbon}

In this section we work with Legendrian ribbons of $\theta-$graphs.  
We examine the relationship between the Legendrian graph and the boundary of its ribbon, the transverse push-off.  
The transverse push-off is another invariant of Legendrian graphs.
We explore whether it contains more information than the classical invariants rotation number and Thurston-Bennequin number.  
We determine the number of components and the self linking number for the push-off of a Legendrian $\theta-$graph.  
In the special case of topologically planar graphs, we prove that the topological type of the transverse push-off of a $\theta-$graph is that of a pretzel-type curve whose coefficients are determined by the Thurston-Bennequin invariant of the graph.\\

Let $g$ be a Legendrian graph.  A \textit{ribbon for $g$} is a compact oriented surface $R_g$ such that:
\be
\item $g$ in contained in the interior of $R_g$;
\item there exists a choice of orientations for $R_g$ and for $\xi$ such that $\xi$ has no negative tangency with $R_g$; 
\item there exists a vector field $X$ on $R_g$ tangent to the characteristic foliation whose time flow $\phi_t$ satisfies 
$\cap_{t\geq 0}\,\phi_t (R_g)=g$; and  
\item the boundary of $R_g$ is transverse to the contact structure.\\
\ee

%The following is a modification of Avdek's procedure \cite{A} to realize the ribbon given a Legendrian graph.  
%In particular, part (e) does not appear in Avdek's construction.  
The following is a construction which takes a graph in the front projection and produces its ribbon viewed in the front projection. 
Portions of this construction were previously examined by Avdek in \cite{A} (algorithm 2, steps 4--6).
%\subsection{Construction of the ribbon surface}
Starting with a front projection of the graph, we construct a ribbon surface containing the graph as described in Figure~\ref{fig-ribbon}.
\be
\item[(a)] to a cusp free portion of an edge we attach a band with a single negative half twist,
\item[(b)] to each left and right cusp along a strand we attach disks containing a positive half twist,
\item[(c,d)] to each vertex we attach twisted disks as in Figure~\ref{fig-ribbon}(c,d),
\item[(e)] crossings in the diagram of the graph are preserved.
\ee

%----- Figure ribbon ----------

\begin{figure}[htpb!]
\begin{center}
\begin{picture}(400, 110)
\put(0,0){\includegraphics{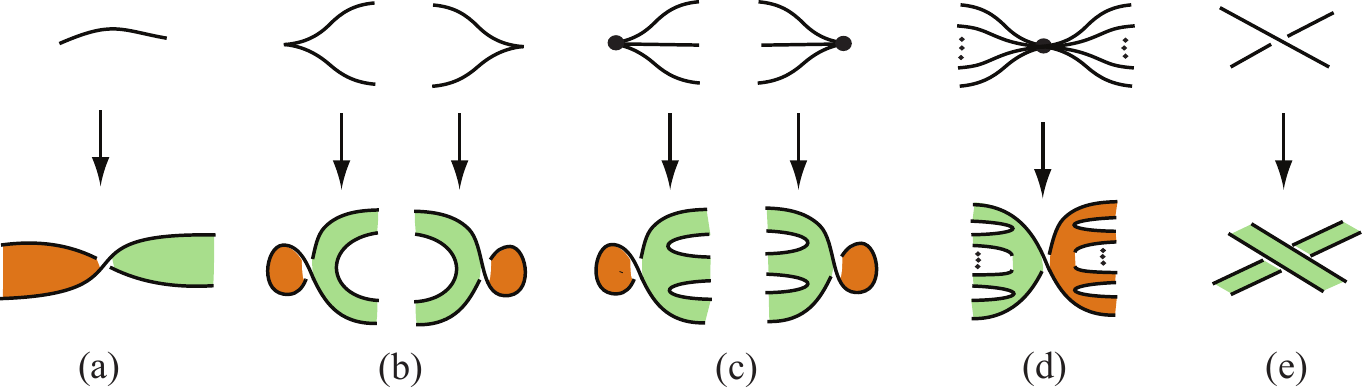}}
\end{picture}
\caption{\small Attaching a ribbon surface to a Legendrian graph. The two sides of the surface are marked by different colors.}
\label{fig-ribbon}
\end{center}
\end{figure}

Legendrian ribbons were first introduced by Giroux \cite{G} to have a well-defined way to contract a contact handlebody onto the Legendrian graph at the core of the handlebody.  
We are interested in some particular features of Legendrian ribbons.  
The boundary of a Legendrian ribbon is an oriented transverse link with the orientation inherited from the ribbon surface.  
The ribbon associated with a given Legendrian graph is unique up to isotopy and therefore gives a natural way to associate a transverse link to the graph.  
\begin{definition} The \textit{transverse push-off} of a Legendrian graph is the boundary of its ribbon.  
\end{definition}
In the case of Legedrian knots the above definition gives a two component link of both the positive and negative transverse push-offs.  
However, with graphs the transverse push-off can have various numbers of components, depending on connectivity and Legendrian type.  
The transverse push-off is a new invariant of Legendrian graphs.  

%-----------------------Self linking---------------------------------
\subsection{Self-linking of transverse push-offs.} 

%In the previous subsection the topological type of the transverse push-off was established.  
Here we determine possible self-linking numbers and the number of components of the transverse push-off of a Legendrian $\theta-$graph.

%----- Figure -----------
\begin{figure}[htpb!]
\begin{center}
\begin{picture}(460, 120)
\put(0,0){\includegraphics{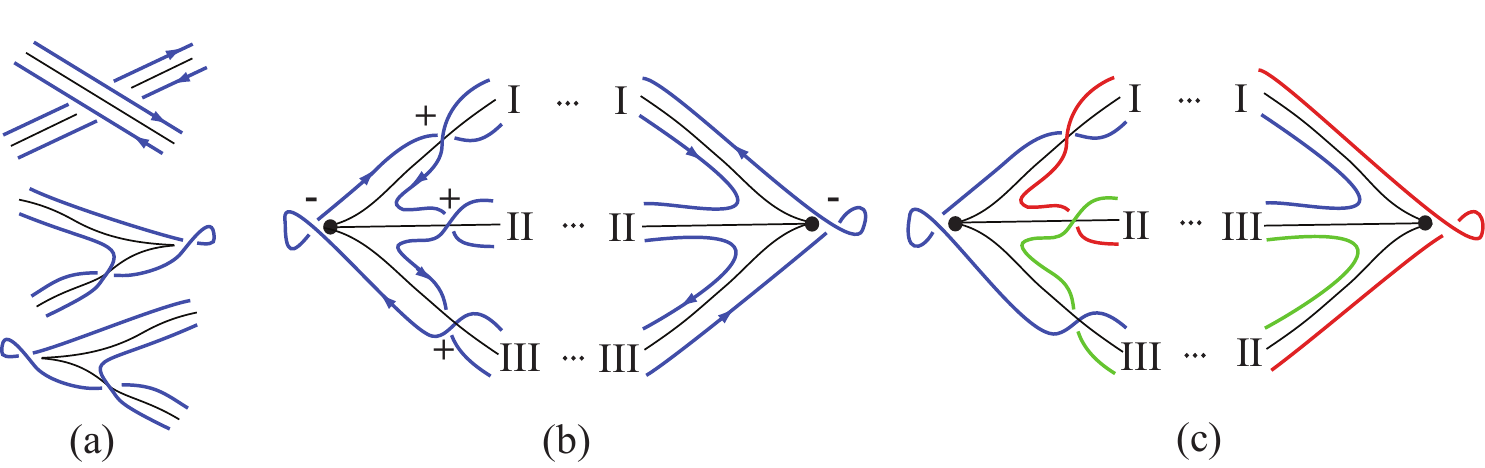}}
\end{picture}
\caption{Transverse push-off of a Legendrian $\theta-$graph with (b) one component or (c) three components.} 
\label{fig-transverse-pushoff}
\end{center}
\end{figure}

\begin{theorem}\label{trans-push-off-sl}
The transverse push-off of a Legendrian $\theta-$graph is either 
 a transverse knot $K$ with $sl=1$ or a three component transverse link whose three components are the transverse push-offs of the three Legendrian cycles given the correct orientation.
\end{theorem} 
\begin{proof}
Before working with the transverse push-off of a Legendrian $\theta-$graph, we will move the graph into a form that will simplify our argument.  
Given an arbitrary Legendrian $\theta-$graph, by Lemma \ref{NearVer}, it can be isotoped to an embedding where near the vertices it has a projection like that shown in Figure \ref{fig-two-vertices}.  
Label the arcs at the left vertex I, II, III from top to bottom.  
Then move the edges around the right vertex (using a combination of RVI and RIV) so that arc I is also in the top position.  
There are two possibilities for the order at the right vertex.  
The first case, where the arcs are I, II, III from top to bottom at the right vertex, shown in Figure \ref{fig-transverse-pushoff}(b), we will call \textit{parallel vertices}.  
The second case, where the arcs are I, III, II from top to bottom at the right vertex, shown in Figure \ref{fig-transverse-pushoff}(c), we will call \textit{antiparallel vertices}.   

Now we will focus on the number of components of the transverse push-off.  
For simplicity of book keeping we will place the negative half twists that occur on each cusp free portion of an edge to the left on that portion of the edge.  
For the projections shown in Figure \ref{fig-transverse-pushoff}(b,c) the portion of the graph not pictured could have any number of crossings and cusps.  
Along each edge, the top (resp. bottom) position of the strands is preserved through cusps and crossings.  
See Figure \ref{fig-transverse-pushoff}(a).  
So we see that the arc of the transverse push-off which lies above (resp. below) the Legendrian arc in the projection on one side of the diagram still lies above (resp. below) on the other side.  
Thus the number of components in the transverse push-off can be determined by a careful tracing of the diagrams in Figure \ref{fig-transverse-pushoff}(b,c).  
Therefore graphs with parallel vertices have a transverse push-off with one component, and graphs with antiparallel vertices have a transverse push-off with three components.

%\tb{The two cases come from a difference in the order of the edges around the two vertices.
%If the cyclic order of edges is the same at the two vertices, the push-off is a knot, while if the cyclic order of edges differs at the two vertices, the push-off is a three component link.
%See Figure~\ref{fig-transverse-pushoff}.}

If the boundary of the Legendrian ribbon is a knot $T$,   then $sl(T)$ equals the signed count of crossings in a front diagram for $T$.
Crossings in the diagram of the graph and cusps along the three edges do not contribute to this count.
A cusp contributes a canceling pair of  positive and negative crossings.
A crossing contributes two negative and two positive crossings. See Figure~\ref{fig-transverse-pushoff}(a). 
Apart from these, there is one positive crossing along each edge and one negative crossing for every disk at each vertex, giving $sl(T)=1$. See Figure  \ref{fig-transverse-pushoff}(b).

If the boundary has three components $T_1$, $T_2$ and $T_3$, then they have the same self linking as the transverse push-offs of the cycles of the Legendrian graph with the correct orientation.  
Let $\bar{\mathcal{C}_i}$ be the cycle $\mathcal{C}_i$ with the opposite orientation.  
Then  $T_1$, $T_2$ and $T_3$, are the positive transverse push-offs of $\bar{\mathcal{C}_1}$,  $\mathcal{C}_2$ and $\bar{\mathcal{C}_3}$, respectively.
 \end{proof}

\subsection{Topologically Planar Legendrian $\theta-$graphs}
%Here we explore the transverse push-off of topologically planar Legendrian $\theta$-graphs.  

To be able to better understand the topological type of a Legendrian ribbon and the transverse 
push-off (its boundary) we will model the ribbon with a flat vertex graph.  
A \textit{flat vertex graph} (or \textit{rigid vertex graph}) is an embedded graph where the vertices are rigid disks with the edges being flexible tubes or strings between the vertices.  
This is in contrast with pliable vertex graphs (or just spatial graphs) where the edges have freedom of motion at the vertices.  
Both flat vertex  and pliable vertex graphs are studied up to ambient isotopy and have  sets of five Reidemeister moves.  
For both of them the first three Reidemeister moves are the same as those for knots and links and Reidemeister move IV consists of moving an edge over or under a vertex.  
See Figure \ref{fig-ReidemeisterIVandV}.  
For flat vertex graphs, Reidemeister move V is the move where the flat vertex is flipped over. 
For pliable vertex graphs, Reidemeister move V is the move where two of the edges are moved near the vertex in such a way that their order around the vertex is changed in the projection.   

For a \textit{Legendrian ribbon, the associated flat vertex graph} is given by the following construction:   a vertex is placed on each twisted disk -- where the original vertices were, and an edge replaces each band in the ribbon.  
The information that is lost with this model is the amount of twisting that occurs on each edge.  
The flat vertex graph model is particularly useful when working with the $\theta-$graph because it is a trivalent graph.  
We see with the following Lemma, the relationship between trivalent flat vertex and trivalent pliable vertex graphs.
 
%----- Figure -----------

\begin{figure}[htpb!]
\begin{center}
\begin{picture}(220, 145)
\put(0,0){\includegraphics{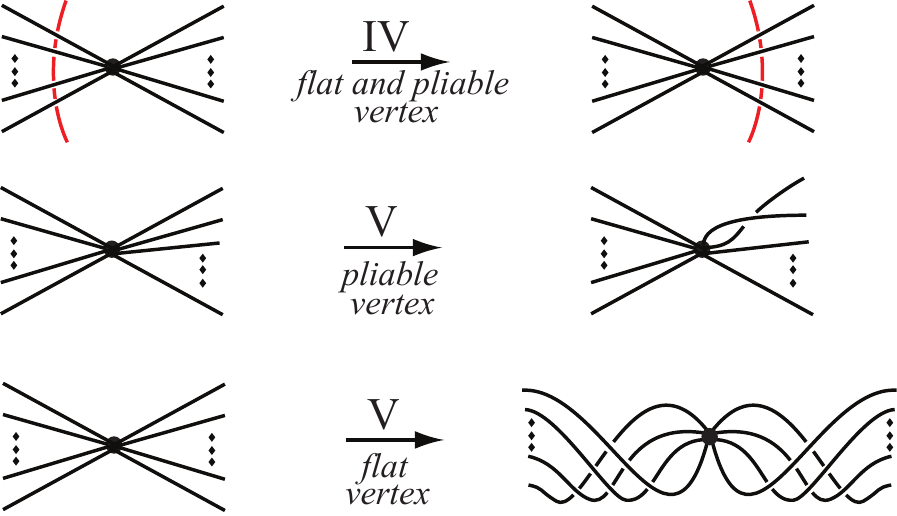}}
\end{picture}
\caption{\small Reidemeister moves IV and V for pliable and flat vertex graphs.}
\label{fig-ReidemeisterIVandV}
\end{center}
\end{figure}
 
\begin{lemma} 
For graphs with all vertices of degree 3 or less, the set of equivalent diagrams is the same for both pliable and flat vertex spatial graphs.  
\label{same-moves}
\end{lemma}
\begin{proof} We follow notation in \cite[pages 699, 704]{K}.  The lemma can be reformulated to say, given the diagrams of two ambient isotopic pliable vertex graphs with maximal degree 3, these are also ambient isotopic as flat vertex graphs, and vice versa.  The Reidemeister moves for pliable vertex graphs and flat vertex graphs differ only in Reidemeister move V.  See Figure~\ref{fig-ReidemeisterIVandV}.  For pliable vertex graphs, Reidemeister move V is the move where two of the edges are moved near the vertex in such a way that this changes their order around the vertex in the projection.   For flat vertex graphs, Reidemeister move V is the move where the flat vertex is flipped over.  For vertices of valence at most 3, these two moves give the same diagrammatic results.  Thus the same sequence of Reidemeister moves can be used in the special case of graphs with maximal degree 3.   
\end{proof}

Here we set up the notation that will be used in the following theorem.  
For a Legendrian $\theta-$graph $G$, we consider a front projection in which the neighborhoods of the two vertices are as those in Figure~\ref{fig-two-vertices} and we denote its three cycles by 
$\C_1$, $\C_2$ and $\C_3$, following the notation of Section 2. 
Let $\textrm{cr}[e_i,e_j]$  be the signed intersection count of edges $e_i$ and $e_j$  in the cycle $\C_1$, $\C_2$ or  $\C_3$ which they determine.
Let $\textrm{cr}[e_i]$  be the signed self-intersection count of $e_i$.
Let $tb_1$, $tb_2$ and $tb_3$ be the Thurston-Bennequin numbers of $\C_1$, $\C_2$ and $\C_3$.

\begin{theorem} Let $G$ represent a topologically planar Legendrian $\theta-$graph with $tb = (tb_1, tb_2, tb_3)$.
Then the boundary of its attached ribbon is  an $(a_1, a_2, a_3)-$pretzel, where 
$a_1= tb_1+tb_2-tb_3$,
$a_2= tb_1+tb_3-tb_2$ and
$a_3= tb_2+tb_3-tb_1$.
\label{pretzel}
\end{theorem}
\begin{proof}
The proof will be done in two parts.  
First, the transverse push-off will be shown to be a pretzel knot or link.  
Second, it will be shown to be of the particular type of pretzel, an $(a_1, a_2, a_3)-$pretzel knot or link.

We first look at the ribbon as a topological object.  
If the ribbon can be moved through ambient isotopy  to a projection where the three bands do not cross over each other and come together along a flat disk, then the boundary of the ribbon would be a pretzel link with crossings only occurring as twists on each band.  
If we model the ribbon with a flat vertex graph this simplifies our question to whether the resulting flat vertex graph can be moved so that it is embedded in the plane.
The resulting graph is topologically planar because it is coming from a topologically planar Legendrian graph.
Thus by Lemma \ref{same-moves}, it can be moved to a planar embedding.  

In order to show the pretzel knot (or link) is an $(a_1, a_2, a_3)-$pretzel, we will look at what happens to the ribbon as the associated flat vertex graph is moved to a planar embedding.   
We will work with the Legendrian $\theta-$graph in the form shown in Figure~\ref{fig-two-vertices} near its vertices.  
We need to count the number of twists in the bands of the Legendrian ribbon once it has been moved to the embedding where the associated flat vertex graph is planar.  
We will prove $a_1= tb_1+tb_2-tb_3$ by writing each of these numbers in terms of the number of cusps and the number of singed crossings between the edges of the Legendrian graph.
The proofs for $a_2$ and $a_3$ are similar.  

We will use the following observations to be able to write $a_1$, the number of half twists in the band associated with edge $e_1,$ in terms of the number of cusps, $\textrm{cr}[e_i] $ and $\textrm{cr}[e_i,e_j].$
\be
\item Based on the construction of the ribbon surface, $c$ cusps on one of the edges contribute with $c+1$ negative half twists to the corresponding band.
\item We look at each of the Reidemeister moves for flat vertex graphs and see how they change the number of twists on the associated band of the ribbon surface.  
%We follow the notation in \cite[pages 699, 704]{K}.

\be
\item A positive (negative) Reidemeister I move adds a full positive (negative) twist to the band. See Figure~\ref{fig-RedemeisterI-V}(a,b).
\item Reidemeister moves II, III and IV do not change the number of twists in any of the bands.
\item A Reidemeister V move adds a half twist on each of the three bands.  See Figure~\ref{fig-RedemeisterI-V}(c,d).
The sign of the half twists depends on the crossing, and which bands are crossed.
  If the bands have a positive (resp. negative) crossing, then they each have the addition of a positive (resp. negative) half twist, and the other band has the addition of a negative (resp. positive) half twist.  
\ee

%---Figure--- R1 for flat vertex graphs-----

\begin{figure}[htpb!]
\begin{center}
\begin{picture}(500, 180)
\put(0,0){\includegraphics[width=6in]{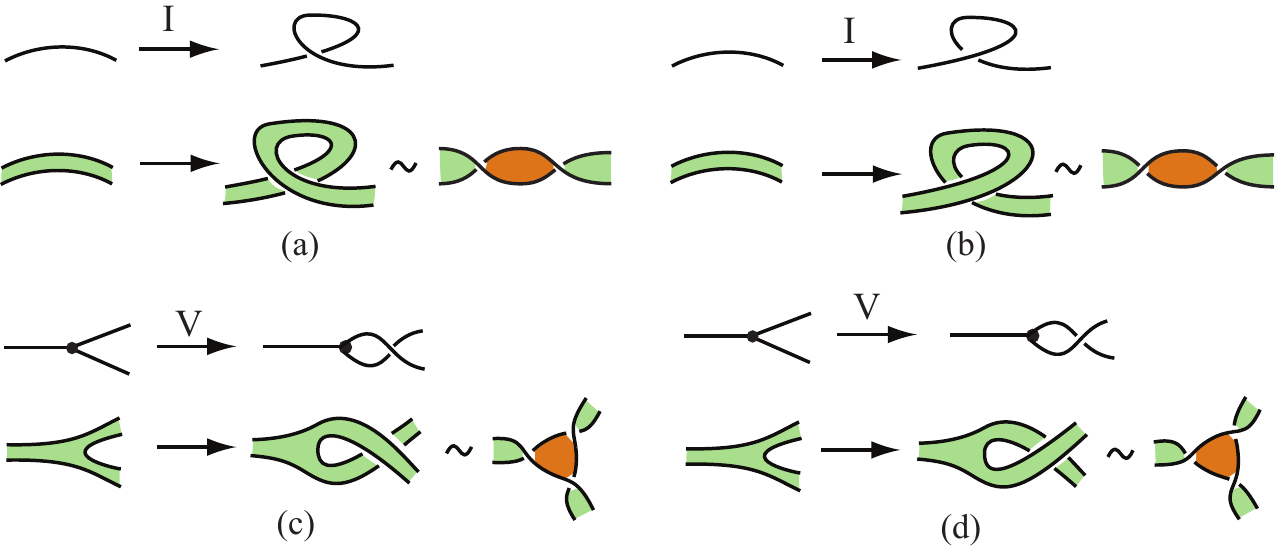}}
\end{picture}
\caption{\small (a) a positive Reidemeister I move adds a full positive twist to the band, (b) a negative Reidemeister I move adds a full negative twist to the band,
(c,d) a  Reidemeister V move adds  a half twist on each of the three bands.
}
\label{fig-RedemeisterI-V}
\end{center}
\end{figure}

\ee

Since we proved earlier that the flat vertex graph can be moved to a planar embedding, we know that all of the crossings between edges will be eventually removed through Reidemeister moves.  
Thus  this gives:
\[ a_1  =  -[\textrm{cusps on }e_1]-1+2\, \textrm{cr}[e_1] +\textrm{cr}[e_1,e_2]+ \textrm{cr}[e_1,e_3] -\textrm{cr}[e_2,e_3]\]
This count is easily seen to be invariant under moves RII and RIII, since these do not change the signed crossing of the diagram. 
We show it is invariant under move RIV at the end of the proof. 

Next, we describe $tb_1+tb_2-tb_3$ in terms of the number of cusps and the crossings between the edges.  
Recall, for a cycle $\C$ we use $$ tb(\C)= w(\C)-\frac{1}{2}\textrm{cusps}(\C).$$ 
Thus,
\begin{eqnarray*}
tb_1+tb_2-tb_3&=&w(\C_1)-\frac{1}{2}\textrm{cusps}(\C_1)+w(\C_2)-\frac{1}{2}\textrm{cusps}(\C_2)-w(\C_3)+\frac{1}{2}\textrm{cusps}(\C_3)\\
&=& \textrm{cr}[e_1, e_2] + \textrm{cr}[e_1] + \textrm{cr}[e_2] -\frac{1}{2}\big([\textrm{cusps on }e_1]+[\textrm{cusps on }e_2]+2\big)\\
& &+\textrm{cr}[e_1, e_3] + \textrm{cr}[e_1] + \textrm{cr}[e_3] -\frac{1}{2}\big([\textrm{cusps on }e_1]+[\textrm{cusps on }e_3]+2\big)\\
& &- \big( \textrm{cr}[e_2, e_3] + \textrm{cr}[e_2] + \textrm{cr}[e_3]\big) +\frac{1}{2}\big([\textrm{cusps on }e_2]+[\textrm{cusps on }e_3]+2\big)\\
&=& -[\textrm{cusps on }e_1]-1 + 2\, \textrm{cr}[e_1] +\textrm{cr}[e_1,e_2]+ \textrm{cr}[e_1,e_3] -\textrm{cr}[e_2,e_3].
\end{eqnarray*}
Thus, $a_1= tb_1+tb_2-tb_3$.

\textit{\textbf{Claim}: The sum $2\, \textrm{cr}[e_1] +\textrm{cr}[e_1,e_2]+ \textrm{cr}[e_1,e_3] -\textrm{cr}[e_2,e_3]$ is unchanged under Reidemeister move IV.  }

\textit{Proof of claim.  }
Let $b_1=2\, \textrm{cr}[e_1] +\textrm{cr}[e_1,e_2]+ \textrm{cr}[e_1,e_3] -\textrm{cr}[e_2,e_3]$. 
Let $d$ represent the strand that is moved past the vertex. 
We distinguish two cases, $(a)$ and $(b)$, according to the number of crossings on each side of the vertex. See Figure~\ref{fig-RademeisterIV}.
We check that the contributions to $b_1$ of the crossing before the move (left) is the same as the contribution to $b_1$ of the crossings after the move (right).
The strand $d$ can be part of $e_1$, $e_2$ or $e_3$.
For  both cases (a) and (b), the equality is shown step by step for $d=e_1$ and $d=e_3$.  In a similar way $b_1$ is unchanged if $d=e_2$.  

%---Figure--- RIV -----
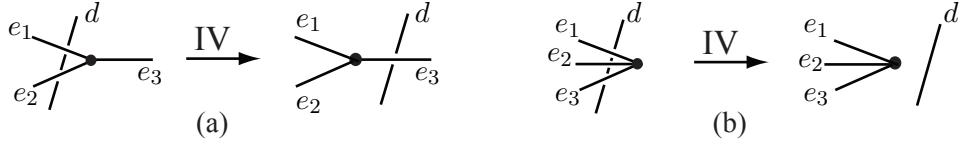
\begin{figure}[htpb!]
\begin{center}
\begin{picture}(340, 50)
\put(0,0){\includegraphics{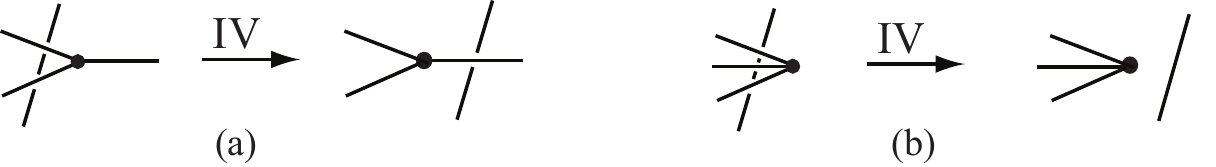}}
\put(-9,40){$e_1$}
\put(-7,15){$e_2$}
\put(20,45){$d$}
\put(40,23){$e_3$}
\put(100,43){$e_1$}
\put(100,12){$e_2$}
\put(145,45){$d$}
\put(145,23){$e_3$}
\put(198,41){$e_1$}
\put(195,30){$e_2$}
\put(198,15){$e_3$}
\put(292,41){$e_1$}
\put(290,28){$e_2$}
\put(292,15){$e_3$}
\put(225,45){$d$}
\put(345,45){$d$}
\end{picture}
\caption{\small Reidemeister IV moves change crossings between different pairs of edges. }
\label{fig-RademeisterIV}
\end{center}
\end{figure}

\noindent Case (a-1) If $d$ is part of $e_1$,  
then $b_{1,\textrm{left}}= 2\textrm{cr}[e_1]+\textrm{cr}[e_1,e_2]=
\textrm{cr}[e_1]$, since the two crossings have opposite sign when seen in the cycle determined by $e_1$ and $e_2$;
and $b_{1,\textrm{right}}= \textrm{cr}[e_1,e_3]= \textrm{cr}[e_1].$\\
Case (a-2) If $d$ is part of $e_3$, then
$b_{1,\textrm{left}}= \textrm{cr}[e_1,e_3]-\textrm{cr}[e_2,e_3]=\textrm{cr}[e_2,e_3]-\textrm{cr}[e_3]=0$ and $b_{1,\textrm{right}}= 0$.\\
\noindent Case (b-1) If $d$ is part of $e_1$,
then $b_{1,\textrm{left}}= 2\textrm{cr}[e_1]+\textrm{cr}[e_1,e_2]+\textrm{cr}[e_1,e_3] = 0$ and $b_{1,\textrm{right}}= 0$. \\
\noindent Case (b-2) If $d$ is part of $e_3$,
then $b_{1,\textrm{left}}=\textrm{cr}[e_1,e_3]-\textrm{cr}[e_2,e_3]=0$, since both these crossings have sign opposite to $\textrm{cr}[e_3]$; and $b_{1,\textrm{right}}= 0$. \\

This complete the proof of the claim and the theorem.  
\end{proof}

The combination of Theorem \ref{trans-push-off-sl} and Theorem \ref{pretzel} gives a complete picture of the possible transverse push-offs of topologically planar Legendrian $\theta-$graphs.  
In this case, the transverse push-off is completely described by the $tb$ of the graph.  
So while this does not add to our ability to distinguish topologically planar Legendrian $\theta-$graphs, it does add to our understanding of the interaction between a Legendrian graph and its transverse push-off.

It is worth noting that Theorem \ref{pretzel} also implies that the transverse push-off will either have one or three components.  
The possible transverse push-offs of a topologically planar Legendrian $\theta$-graph are more restricted than it may first appear.  
Not all pretzel links will occur in this way.  
In Theorem \ref{pretzel}, we found the pretzel coefficients as linear combinations with coefficients +1 or $-1$ of the $tb$'s.
We note that the three pretzel coefficients  have the same parity, restricting the number of components the transverse push-off can have.
If exactly one of or all three of $tb_1$, $tb_2$ and $tb_3$ are odd, then all pretzel coefficients are odd and the pretzel curve is a knot.
If none or exactly two of $tb_1$, $tb_2$ and $tb_3$  are odd, then  all pretzel coefficients are even and the pretzel curve is a three component link.
The pairwise linking between its components is equal to the number of full twists between the corresponding  pair of  strands in the pretzel presentation, i.e. $a_1/2$, $a_2/2$ and $a_3/2$.
%So there are two possibilities, the transverse push-off is either a knot or a three component link with pairwise linking numbers $a_1/2$, $a_2/2$ and $a_3/2$.

%------------subsection-----------------------------------------------------------------
\subsection{The transverse push-off of  $n\theta-$graphs.}

We give examples showing the boundary of the Legendrian ribbon associated to an  $n\theta-$graph, $n>3$, is not necessarily a pretzel-type link. 
Independent of $n$, each component of an $n-$pretzel type link is linked with at most two other components. 
The tranverse push-offs of the graphs in Figure~\ref{fig-ntheta} have at least one component linking more than two other components of the link.
The characterization as a pretzel curve of the topological type of the push-off is therefore exclusive to the case $n=3$, that of $\theta-$graphs.
%----- Figure -----------
\begin{figure}[htpb!]
\begin{center}
\begin{picture}(370, 230)
\put(0,0){\includegraphics{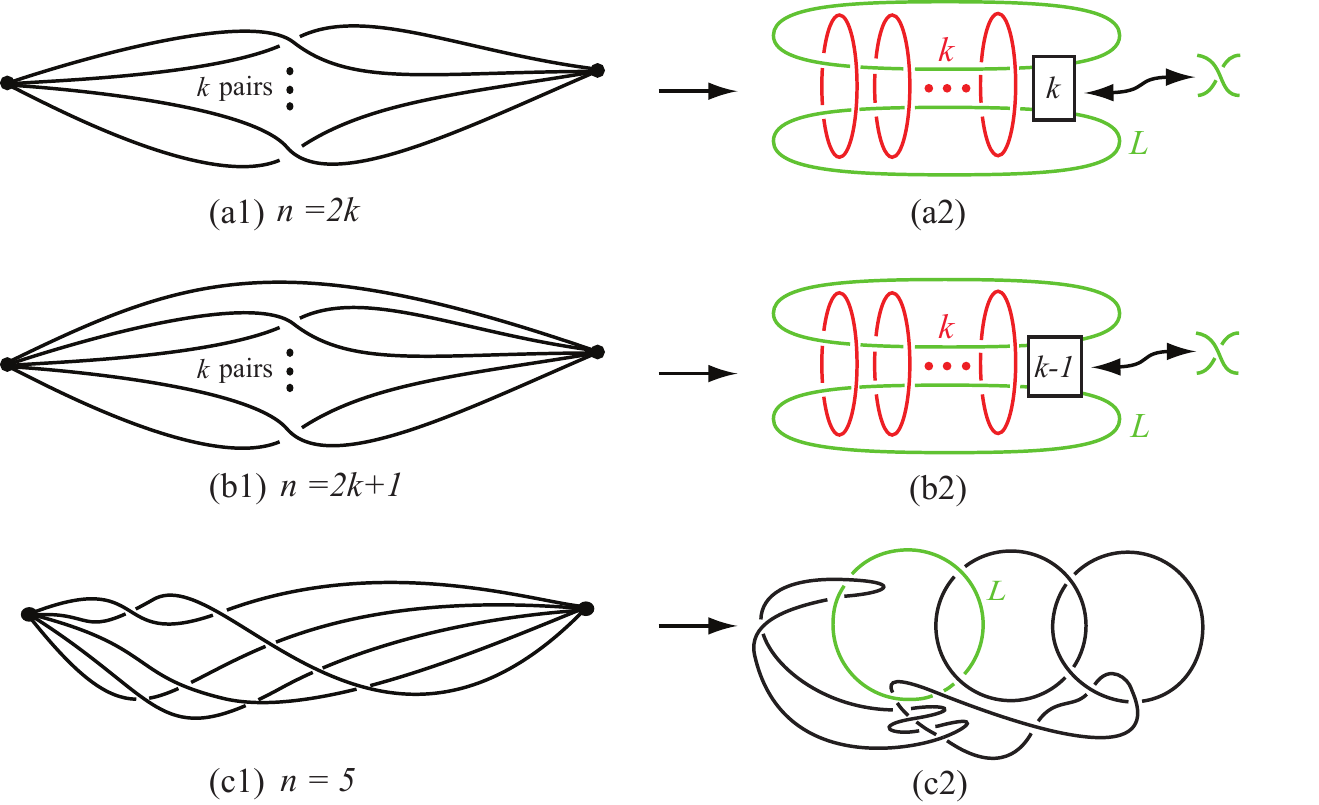}}
\put(267,190){\tr{$L_k$}}
\put(267,110){\tr{$L_k$}}
\end{picture}
\caption{\small The $n\theta-$graphs in (a1), (b1) and (c1) have transverse push-offs (a2), (b2) and (c2) which do not have the topological type of a pretzel-type curve.}
\label{fig-ntheta}
\end{center}
\end{figure}

For $n=2k, k\ge 2$, let $L_{2k}$ be the Legendrian $2k\theta-$graph whose front projection is the one in Figure~\ref{fig-ntheta}(a1). 
Then the transverse push-off has the topological type of the link $L\cup L_k$ in Figure~\ref{fig-ntheta}(a2).
If $k$ is odd, $L$ has one component and it links all $k\ge 3$ components of $L_k$.
If $k$ is even, then $L$ has two components where each of the two components links all $k\ge 2$ components of $L_k$ and the other component of $L$.

For $n=2k+1, k\ge 3$, let $L_{2k+1}$ be the Legendrian $(2k+1)\theta-$graph whose front projection is the one in Figure~\ref{fig-ntheta}(b1). 
Then the transverse push-off has the topological type of the link $L\cup L_k$ in Figure~\ref{fig-ntheta}(b2).
If $k$ is even, then $L$ has one component and it links all $k\ge 3$ components of $L_k$.
If $k$ is odd, then $L$ has two components where each of the two components links all $k\ge 3$ components of $L_k$ and the other component of $L$.

For $n=5$, the link in Figure~\ref{fig-ntheta}(b2) is a pretzel link and we give a different example in this case, the one in Figure~\ref{fig-ntheta}(c1,c2).
The highlighted component of the transverse push-off links three other components.

\newpage
\bibliographystyle{amsplain}

\end{document}